\documentclass[a4paper,11pt]{amsart}
\usepackage{amssymb,amsthm,fullpage,amsmath,psfrag,hyperref}
\usepackage[french,english]{babel}

\newtheorem{theorem}{Theorem}[section]

\newtheorem{proposition}[theorem]{Proposition}

\theoremstyle{definition}
\newtheorem{definition}[theorem]{Definition}

\def\BbR{\mathbb{R}}
\def\BbN{\mathbb{N}}

\def\R{\BbR}
\def\S{\mathbb{S}}
\def\N{\BbN}
\usepackage{psfrag}
\usepackage{color, epsfig}
\theoremstyle{remark}
\newtheorem{remark}[theorem]{Remark}
\numberwithin{equation}{section}
\renewcommand\div{\text{div}\,}
\def\curl{\text{curl}\,}
\def\eps{\varepsilon}
\def\Div{{\rm div}\,}
\usepackage{mathrsfs}

\title{Lagrangian controllability at low Reynolds number}

\author{O. Glass}
\address{CEREMADE, Universit\'e Paris-Dauphine \& CNRS, PSL, Place du Mar\'echal de Lattre de Tassigny, 75775 Paris Cedex 16, France}

\author{T. Horsin}
\address{Conservatoire National des Arts et M\'etiers, M2N, Case 2D 5000, 292 rue Saint-Martin, 75003 Paris, France}
\date{}

\dedicatory{}
\begin{document}
\maketitle

\begin{abstract}
In this paper, we establish a result of Lagrangian controllability for a fluid at low Reynolds number, driven by the stationary Stokes equation.
This amounts to the possibility of displacing a part of a fluid from one zone to another by suitably using a boundary control.
This relies on a weak variant of the Runge-Walsh's theorem (on approximation of harmonic functions) concerning the Stokes equation.
We give two variants of this result, one of which we believe to be better adapted to numerical simulations.
\end{abstract}

%
%
%
%
%
%
\section{Introduction}
In this paper, we consider smooth solutions to the stationary Stokes equation in a bounded regular domain $\Omega$ of $\BbR^{N}$, with $N=2,3$:
\begin{subequations}\label{eq:stokes}
\begin{gather}
-\Delta u + \nabla p = 0 \hbox{ in }\Omega, \label{eq:stokes.1} \\
\div u = 0 \hbox{ in } \Omega, \label{eq:stokes.2} \\
\int_\Omega p \, dx=0. \label{eq:stokes.3}
\end{gather}
\end{subequations}
This system is standard to describe incompressible and highly viscous fluids; here $u: \Omega \rightarrow \R^{N}$ represents the velocity field, and $p: \Omega \rightarrow \R$ the pressure. \par
\ \par
The main questions addressed in this paper are of two forms. \par
\ \par
The first one is the problem of approximation of a solution of \eqref{eq:stokes} defined on the domain $\Omega$, by solutions of the same system, but defined on a larger set.
To be more precise, the question that we raise is the following: given $\Omega'$ an open set in $\BbR^N$ such that $\overline{\Omega}\subset \Omega'$  and $K$ a compact subset of $\Omega$, $k\in \BbN$, $\alpha\in (0,1)$, and given $(u,p)$ a regular solution of  \eqref{eq:stokes} and $\eps>0$, is it possible to find $(\overline{u},\overline{p})$ a solution of the equations \eqref{eq:stokes} on $\Omega'$ such that
\begin{equation*}
\| (u,p) - (\overline{u},\overline{p}) \|_{C^{k,\alpha}(K)} \leq \eps ?
\end{equation*}
This question, to which we give a partial positive answer, is related to the famous Runge theorem concerning the approximation of holomorphic functions by rational functions as well as to its extension by Walsh (see e.g. \cite{GARD}) to the case of harmonic functions on a open set in $\BbR^N$. The former has been used in \cite{GlHo08} to obtain a result of approximate Lagrangian controllability of the incompressible Euler equation in dimension $2$, whereas the latter has been used in \cite{GlHo10} to get a result in dimension $3$. Both results can be considered as the cornerstone of the known approaches to Lagrangian controllability (though the complete method require other technical results). \par
\ \par
Our second main question here deals precisely with the Lagrangian controllability itself in the framework of the Stokes model.
Given a fluid model, achieving the Lagrangian controllability between two subsets consists in being able to act on a given part of the domain in such a way that the resulting fluid flow maps one subset on the other in a given time. We moreover require that in between the fluid particles of the concerned subset do not leave the domain. \par
Let us be more specific on the problem under view. Let $\Omega$ a smooth bounded domain, and let $\Sigma$ an open nonempty part of its boundary.
Here we consider \eqref{eq:stokes} in quasi-static form, that is, the solution is time-dependent even if the driving equation is stationary. Moreover, we consider the system controlled from $\Sigma$, that is, we assume that we can prescribe the boundary conditions on $\Sigma$. On $\partial \Omega \setminus \Sigma$, on the contrary, this is constrained by the standard Dirichlet boundary conditions. So we write the system as follows:
\begin{subequations} \label{eq:stokes2}
\begin{gather}
\label{eq:stokes2.1} 
-\Delta u + \nabla p = 0\mbox{ in }(0,T)\times \Omega, \\
\label{eq:stokes2.2}
\div u = 0 \mbox{ in } (0,T)\times \Omega, \\
\label{eq:stokes2.3} 
\int_\Omega p \, dx = 0,\,\mbox{ on }(0,T), \\
\label{eq:stokes2.4}
u = 0 \mbox{ on } (0,T) \times (\partial \Omega \setminus \Sigma).
\end{gather}
\end{subequations}
With this form, the Lagrangian controllability raises the question of the possibility given two sets, to find a suitable control on $ [0,T] \times \Sigma$, so that the flow associated to the solution $u$ sends the first set on the second one. There are naturally constraints on the two sets: they should have the same area/volume (due to the incompressibility of the model), and we add a constraint on their form (namely, that they are smooth Jordan curves/surfaces). Moreover, one can relax this notion of Lagrangian controllability to the {\it approximate} Lagrangian controllability by asking the first set to be merely sent arbitrarily close to the target, rather than reaching it exactly. \par
Note that System~\eqref{eq:stokes2} is a good approximation of the standard Navier-Stokes system
\begin{subequations} \label{eq:navier-stokes}
\begin{gather}
\partial_t u+(u \cdot \nabla) u- \nu \Delta u+\nabla p=0\label{eq:navier-stokes.1} \\
\div u =0.
\end{gather}
\end{subequations}
when the Reynold's number $R_e$ tends to $0^+$, that is, when $\nu \rightarrow +\infty$. This model has already been studied in the framework of controllability in different situations. In \cite{JSMTTMT07}, a model of very small species moving using flagella like is studied. The localization of the action is used to model the fluid as in \eqref{eq:stokes2} but on a time dependent set.  Let us also mention the works \cite{Alouges-al}, \cite{Giraldi-al} and \cite{Loheac-Munnier}, where the problem of swimming in a Stokes fluid which is considered to be quasi-static is considered. \par
The assumption of a very small Reynold's number means that the inertia is neglected compared to viscosity forces. We have already obtained Lagrangian controllability results in the opposite case of infinite Reynold's number which corresponds to the inviscid case driven by the Euler equation (see \cite{GlHo08,GlHo10}). However in the intermediate case which would be described by Navier-Stokes equation \eqref{eq:navier-stokes}, the Lagrangian controllability remains an open problem. \par
%

%
%
%
%
%
\section{Main results}
The first main result of this paper, concerning the partial extension of Walsh's theorem to the Stokes equation, is an approximation result for solutions of equation \eqref{eq:stokes}. It is an adaptation of Walsh's theorem of harmonic approximation. For the reader's convenience, we recall Walsh's theorem or its following variant, see \cite[Theorem 1.7]{GARD}.
\begin{theorem}[Walsh, Gardiner] \label{th:approxharm1}
Let ${\mathcal O}$ be an open set in $\R^{N}$ and let $K$ be a compact set in $\R^{N}$
such that that ${\mathcal O}^{*} \setminus K$ is connected, where ${\mathcal O}^{*}$ is the Alexandroff compactification of ${\mathcal O}$. Then, for each  function $u$ which is harmonic on an open set containing $K$ and each $\varepsilon>0$, there is a harmonic function $v$ in ${\mathcal O}$
such that $\|v - u\|_{\infty} < \varepsilon$ on $K$.
\end{theorem}
A typical situation is when ${\mathcal O}=\R^{N}$ and $K$ is contractile. Then one can approximate on $K$ a harmonic function defined on a neighborhood of $K$ by harmonic polynomials. Here we prove a variant for the Stokes equation, where the approximation is defined only on a bounded open set, not the whole space. Precisely, we establish the following.
\begin{theorem}\label{Thm:RWforS}
Let $K$ a compact set in $\R^{N}$, $N=2$ or $3$. Let ${\mathcal V}$ and $\Omega$ two bounded open sets such that $K \subset {\mathcal V}$, $\overline{\mathcal V} \subset \Omega$ and each connected component of $\R^{N} \setminus K$ contains an interior point of $\R^{N} \setminus \Omega$. Then for any solution $(u,p) \in C^{\infty}({\mathcal V};\R^{N+1})$ of the Stokes equation in ${\mathcal V}$:
\begin{equation} \label{Eq:StokesOmega}
\left\{ \begin{array}{l}
- \Delta u + \nabla p = 0, \\
\Div u =0
\end{array} \right.
\end{equation}
for any $k \in \N$ and any $\varepsilon>0$ there exists $(\overline{u},\overline{p}) \in C_{c}^{\infty}(\R^{N}; \R^{N+1})$ a solution of the Stokes equation in $\Omega$:
\begin{equation} \label{Eq:StokesBR}
\left\{ \begin{array}{l}
- \Delta \overline{u} + \nabla \overline{p} = 0, \\
\Div \overline{u} =0
\end{array} \right.
\text{ in } \Omega,
\end{equation}
and
\begin{equation} \label{Eq:Approx}
\| \overline{u} - u \|_{C^{k}(K)} \leq \varepsilon.
\end{equation}
\end{theorem}
Now we state our main application of this result, namely an approximate Lagrangian controllability result for the Stokes equation. 
For that, we first recall some topological and geometrical notions.
\begin{definition}
A regular Jordan curve in $\BbR^2$ is the image of $\mathbb{S}^1$ by a diffeomorphism, and similarly a Jordan surface in $\BbR^3$ is the image of $\mathbb{S}^2$ by a diffeomorphism.
\end{definition}
According to the Jordan's theorem and the Jordan-Brouwer's theorem, the complement of a Jordan curve/surface $\gamma$ in $\BbR^{N}$, $N=2$ or $3$, defines two connected components; the bounded one will be denoted $\hbox{int}(\gamma)$. The exterior normal to Jordan curves/surfaces will be denoted by $\nu$ where $n$ will be preferred in the case of the boundary of $\Omega$. \par
\ \par
Also, we recall the following definitions.
\begin{definition}
Two Jordan curves $\gamma_0$ and $\gamma_1$ are said to be homotopic in $\Omega$, if there exists a continuous map $h:[0,1]\times \mathbb{S}^1\to \Omega$ such that 
$h(0,\cdot)$ is a parameterization of $\gamma_0$ and $h(1,\cdot)$ is a parameterization of $\gamma_1$.
\end{definition}
\begin{definition}
Two Jordan surfaces $\gamma_0$ and $\gamma_1$ embedded in $\R^3$ are said to be isotopic in $\Omega$, if there exists a continuous map $h:[0,1]\times \S^2\to \Omega$ (called isotopy) such that $h(0)=\gamma_{0}$, $h(1)=\gamma_{1}$ and for each $t\in [0,1],\ h(t,\cdot)$ is an homeomorphism of $\S^2$ into its image. When, for some $k\in \N \setminus \{0\}$, this homeomorphism is a $C^{k}$-diffeomorphism with respect the space variable, $h$ will be said to be a $C^k$-isotopy, or, when $k=\infty$, a smooth isotopy. 
\end{definition}
Now we introduce the following assumption.
\begin{definition}
 We will say that $\gamma_0$ and $\gamma_1$ two Jordan curves/surfaces satisfy the property $\mathfrak{P}$ if, when $N=2$, $\gamma_0$ and $\gamma_1$ are homotopic in $\Omega$, and if, when $N=3$, $\gamma_0$ and $\gamma_1$ are isotopic in $\Omega$. 	
\end{definition}
Note that when $N=2$, two homotopic Jordan curves in $\Omega$ are in fact isotopic in $\Omega$, see \cite{epstein}. \par
\ \par
We will also use the following notation for a flow: given a suitably regular vector field $u$, we denote by $\phi^u$ the flow of $u$, defined (when possible) by 
\begin{equation} \label{eq:flot}
\partial_t\phi^u(t,s,x) = u(t,\phi^u(t,s,x)) \ \text{ and } \ u(s,s,x)=x.
\end{equation}
\ \par
Our result is then as follows.
\begin{theorem} \label{th:2d3d}
Let $\Omega$ be a smooth bounded open connected set in $\R^{N}$ and $\Sigma$ a nonempty open part of $\partial \Omega$.
If $\gamma_0$ and $\gamma_1$ are smooth Jordan curves/surface in $\Omega$ satisfying the property $\mathfrak{P}$, then, for any $\eps>0$, $k\in \BbN^*$, $\alpha\in (0,1)$, there exists a solution $(u,p) \in C^{\infty}([0,1] \times \overline{\Omega};\R^{N+1})$ of \eqref{eq:stokes2} such that
\begin{subequations}
\begin{align}
\label{eq:approxlagrcontr2d}
\| \phi^u(1,0,\gamma_0) - \gamma_1 \|_{C^{k,\alpha}(\mathbb{S}^{N-1})} \leq \eps, \\
\label{eq:approxlagrcontr2d.1}
\forall t \in [0,1], \ \phi^u(t,0,\gamma_0)\subset \Omega.
\end{align}
\end{subequations}
\end{theorem}
What we mean by \eqref{eq:approxlagrcontr2d}, is that we can find a parameterization of $\gamma_0$ and $\gamma_1$ such that \eqref{eq:approxlagrcontr2d.1} holds. Naturally, for this quasi-static model, the time of controllability $T$ plays no role, and one can reduce the situation to $T=1$ by a simple change of time variable. 
\begin{remark}

If $\gamma_0$ is a real analytic surface, and if for all $t\in [0,1]$, $\phi^u(t,0,\gamma_0)\subset \Omega$, then by using the regularity of the solutions of the Stokes system, one deduces that for all $t \in [0,1]$, $\phi^u(t,0,\gamma_0)$ is also real analytic for all $t$. Thus exact Lagrangian controllability does not hold in general in this framework (pick $\gamma_{1}$ smooth but not analytic).
\end{remark}
%
%
%
%
%
%
%
%
%

\section{Proof of Theorem \ref{Thm:RWforS}}
In this section, we establish Theorem~\ref{Thm:RWforS}. \par
\ \par
\noindent
{\bf Step 0. Geometrical simplification.} In a first step, we show how to reduce to the case where $K$ and ${\mathcal V}$ have smooth boundaries and $\Omega$ is of the form $B(0,R) \setminus \cup_{i=1}^{n} B(A_{i},r_{i})$ with large $R>0$ and small $r_{i}>0$. \par
\ \par
\noindent
$\bullet$ Since $\Omega$ is bounded, it is included in some ball $B(0,R/2)$, $R$ large enough.
In each connected component of $B(0,R) \setminus K$, we pick a unique point $A_{i} \in \overset{\circ}{\R^{N} \setminus \Omega}$ (which is clearly possible due to the assumption). 
We introduce $r_{i}>0$ such that $\overline{B}(A_{i},r_{i}) \subset \R^{N} \setminus \Omega$. In particular, note that $B(A_{i},r_{i})$ does not meet ${\mathcal V}$. 
Of course, since each of such connected component contains a point with rational coordinates, the family $(A_{i})$ is an at most countable.
Clearly one has $\Omega \subset \tilde{\Omega} := B(0,R) \setminus \cup_{i} \overline{B}( A_{i},r_{i})$.
Finally, we add to $K$ all the connected components of $B(0,R) \setminus K$ included in ${\mathcal V}$: call $\tilde{K}$ the resulting set which is still compact (it is bounded and contains its boundary). Now if we can establish the result on $\tilde{K}$, ${\mathcal V}$ and $\tilde{\Omega}$ (that is, starting from a solution in ${\mathcal V}$, approximate it on $\tilde{K}$ by a solution in $\tilde{\Omega}$), it is all the more true on $K$, ${\mathcal V}$ and $\Omega$. \par
\ \par
\noindent
$\bullet$ We claim that $B(0,R) \setminus \tilde{K}$ has a finite number of connected components. If not, consider the subsequence of the points $A_{i}$ associated to the connected components in $B(0,R) \setminus \tilde{K}$. Write it $(A_{i})_{i \in \N}$ again. Then up to a subsequence that we still denote $(A_{i})$, one has $A_{i} \rightarrow A$ as $i \rightarrow +\infty$. Then $A$ cannot belong to a connected component of $B(0,R) \setminus \tilde{K}$, since otherwise, there would be a whole ball around $A$ included in this connected component, and each of these connected components is only visited once by the sequence $A_{i}$. It follows that $A \in \tilde{K} \subset {\mathcal V}$. This means that for $i$ large, $A_{i} \in {\mathcal V}$ which contradicts the fact that $B(A_{i},r_{i})$ does not meet ${\mathcal V}$. \par
We call $A_{1}, \dots, A_{n}$  the remaining points $A_{i}$ corresponding to the connected components of $B(0,R) \setminus \tilde{K}$. \par
\ \par
\noindent
$\bullet$
Now by a celebrated theorem of Whitney, there exists a function $\varphi \in C^{\infty}(\overline{B}(0,R);\R)$ such that $\tilde{K}= \varphi^{-1}(\{0\})$. Replacing $\varphi$ by $\varphi^{2}$ is necessary, we may assume that $\varphi \geq 0$. Due to the compactness of $\tilde{K}$, we have for some $c>0$:
\begin{equation*}
\varphi \geq c \ \text{ in } \ \overline{B}(0,R) \setminus {\mathcal V}.
\end{equation*}
Consider $\delta \in [0,c)$.
Due to Sard's theorem, there exists regular values $\lambda$, $\mu$ and $\nu$ of the function $\varphi$ respectively in in $[\delta/10, \delta /5]$, $[3 \delta/10,2 \delta/5]$ and $[\delta/2,3 \delta/10]$. Call $K_{\delta}:= \varphi^{-1}([0,\lambda])$, ${\mathcal V}^{1}_{\delta}:=\varphi^{-1}([0,\mu))$ and ${\mathcal V}_{\delta}:=\varphi^{-1}([0,\nu))$. All the three subsets have consequently smooth boundaries. Furthermore, reducing $\delta$ if necessary, we can ensure that $\cup_{i=1}^{n} \overline{B}(A_{i}, r_{i}) \cap {\mathcal V}_{\delta} = \emptyset$.
Moreover:
\begin{itemize}
\item if there are connected components of $K_{\delta}$ that do not meet $K$, remove them from $K_{\delta}$,
\item if there are connected components of $B(0,R) \setminus {\mathcal V}^1_{\delta}$ and $B(0,R) \setminus {\mathcal V}_{\delta}$ that do not contain a point $A_{i}$, add them to ${\mathcal V}^1_{\delta}$ (resp. ${\mathcal V}_{\delta}$).
\end{itemize}
Note that since ${\mathcal V}^{1}_{\delta}$ and ${\mathcal V}_{\delta}$ are included in ${\mathcal V}$, a connected component of $B(0,R) \setminus {\mathcal V}^1_{\delta}$ (resp. $B(0,R) \setminus {\mathcal V}_{\delta}$) either contains a connected component of $B(0,R) \setminus {\mathcal V}$ (and hence a point $A_{i}$) or is included in ${\mathcal V}$. Consequently ${\mathcal V}^1_{\delta}$ and ${\mathcal V}_{\delta}$ obtained in this way are included in ${\mathcal V}$. \par
Now if we can establish the result on $K_{\delta}$, ${\mathcal V}_{\delta}$ and $\hat{\Omega}:= B(0,R) \setminus \cup_{i=1}^{n} \overline{B}(A_{i}, r_{i})$, it is all the more true on $K$, ${\mathcal V}$ and $\Omega$. From now on, we write $K$, ${\mathcal V}$ and $\Omega$ for $K_{\delta}$, ${\mathcal V}_{\delta}$ and $\hat{\Omega}$, to simplify the notation. Note in passing that we have obtained ${\mathcal V}_{1}$ (the new notation for ${\mathcal V}^{1}_{\delta}$) with smooth boundary so that
\begin{equation*}
K \subset {\mathcal V}_{1} \subset \overline{{\mathcal V}_{1}} \subset {\mathcal V} \subset \overline{{\mathcal V}} \subset \Omega.
\end{equation*}
\ \par
Let us now go back to the proof of Theorem~\ref{Thm:RWforS}. We consider $(u,p)$ a smooth solution of \eqref{Eq:StokesBR} on ${\mathcal V}$, and want a solution of \eqref{Eq:StokesBR} on $\Omega$ approximating well $u$ on $K$. We follow several consecutive steps. In the sequel, $C>0$ is a constant that may change from line to line and depends on the geometry, but not on the function $u$ or $\varepsilon$. \par
\ \par
\noindent
{\bf Step 1.} First we introduce an extension $\tilde{p} \in C^{\infty}_{c}(\Omega;\R)$ of $p_{|\overline{{\mathcal V}}}$. Next we define
\begin{equation} \label{Eq:DefHatu}
\hat{u}:= u - \nabla \Delta^{-1} \tilde{p} \ \text{ in } \ {\mathcal V},
\end{equation}
where $\Delta^{-1} := \cdot * G$ is the convolution operator with the fundamental solution associated to the Laplacian, that is,
\begin{equation*}
G(x) = \frac{1}{2\pi} \ln |x| \text{ if } N=2, \ \ G(x) = -\frac{1}{4\pi |x|} \text{ if } N=3.
\end{equation*}
Then one can easily check that
\begin{equation} \label{Eq:PropHatu}
\curl \hat{u} = \curl u \ \text{ in } {\mathcal V}, \ \ \ 
\Delta \hat{u} = 0 \ \text{ in } {\mathcal V}.
\end{equation}
Let $\gamma_{1}$, \dots, $\gamma_{g}$ a (minimal) family of smooth loops generating the fundamental group of ${\mathcal V}_{1}$.
One has moreover
\begin{equation} \label{Eq:PropHatu2}
\oint_{\gamma_{i}} (\hat{u} - u) \cdot \tau =0, \  \ \forall i =1, \dots,g.
\end{equation}
\ \par
\noindent
{\bf Step 2.} Next we use Walsh's theorem on $\hat{u}$ which is harmonic in ${\mathcal V}$. We have points $A_{1},\dots, A_{n}$ outside of $\Omega$ in each connected component of $\R^{N} \setminus {\mathcal V}_{1}$. Hence we obtain a harmonic function $\tilde{u}$, defined in $\R^{N} \setminus \{ A_{1},\dots, A_{n} \}$ and such that:
\begin{eqnarray}
\label{Eq:uepsilon1}
\Delta \tilde{u} =0 \text{ in } \R^{N} \setminus \{ A_{1}, \dots, A_{n}\} , \\
\label{Eq:uepsilon2}
\| \hat{u} - \tilde{u} \|_{C^{k+1,\alpha}(\overline{{\mathcal V}_{1}})} \leq \varepsilon,
\end{eqnarray}
where we fixed some $\alpha \in (0,1)$.
In particular we infer with \eqref{Eq:DefHatu} that for some $C>0$,
\begin{equation} \label{Eq:uepsilon3}
\left| \oint_{\gamma_{i}} (u - \tilde{u}) \cdot \tau \right| \leq C \varepsilon \ \text{ and } \ 
\| \curl u - \curl \tilde{u} \|_{C^{k,\alpha}(\overline{{\mathcal V}_{1}})} \leq C \varepsilon.
\end{equation}
\ \par
\noindent
{\bf Step 3.} Now we extend $\curl (u - \tilde{u})_{|\overline{{\mathcal V}_{1}}}$ to a function $w \in C^{k,\alpha}_{c}(\R^{N})$ such that 
\begin{equation*}
\| w \|_{C_{c}^{k,\alpha}(\R^{N})} \leq C \varepsilon.
\end{equation*}
Now, for instance by using Biot-Savart's formula, there exists $v \in C^{k+1,\alpha}(\R^{N})$ such that for a constant $C>0$ depending on $R$:
\begin{equation*}
\curl v = w, \ \div v=0, \ \ \| v \|_{C^{k+1,\alpha}(\overline{B}(0,R))} \leq C \varepsilon.
\end{equation*}
Then $\curl v= \curl (u - \tilde{u})$ on ${\mathcal V}_{1}$, so one has
\begin{equation*}
v = u - \tilde{u} + \nabla \theta + H,
\end{equation*}
for some $\theta \in C^{k+2,\alpha}({\mathcal V}_{1};\R)$ and $H$ some function in the first de Rham (tangent) cohomology space of ${\mathcal V}_{1}$ which one can represent in terms of vector fields by
\begin{equation*}
{\mathcal H}^{1}({\mathcal V}_{1}) := \{ u \in C^{\infty}(\overline{\mathcal V}_{1};\R^{N}) \ / \ \curl u=0 \text{ and } \div u =0 \text{ in } {\mathcal V}_{1}, \ u \cdot n = 0 \text{ on } \partial {\mathcal V}_{1} \}.
\end{equation*}
See for instance \cite[Appendix I]{temam09}. This vector space is of finite dimension $g$. Its basic property is: any $\curl$-free vector field in ${\mathcal V}_{1}$ is the sum of a gradient and an element of ${\mathcal H}^{1}({\mathcal V}_{1})$. \par
We extend $\theta$ to some function in $C^{k+2,\alpha}_{c}(\R^{N};\R)$ arbitrarily.
In particular, we have
\begin{equation} \label{Eq:uepsilon4}
\| u - \tilde{u} +\nabla \theta + H \|_{C^{k+1,\alpha}({\mathcal V}_{1})} \leq C \varepsilon.
\end{equation}
From \eqref{Eq:uepsilon3} and  \eqref{Eq:uepsilon4} we deduce that for all $i=1,\dots,g$, one has
\begin{equation} \label{Eq:H1}
\left| \oint_{\gamma_{i}} H \cdot \tau \right| \leq C \varepsilon.
\end{equation}
Since on ${\mathcal H}^{1}({\mathcal V}_{1})$ the map $H \mapsto \left(\oint_{\gamma_{i}} H \cdot \tau \right)_{i=1,\dots,g}$ is an isomorphism, we deduce finally that
\begin{equation*}
\| H \|_{C^{k+1,\alpha}({\mathcal V}_{1})} \leq C \varepsilon,
\end{equation*}
and consequently that
\begin{equation} \label{H2}
\| u - \tilde{u} +\nabla \theta \|_{C^{k+1,\alpha}({\mathcal V}_{1})} \leq C \varepsilon.
\end{equation}
\ \par
\noindent
{\bf Step 4.} Now we claim that there exists $\beta \in C^{k+1,\alpha}(\Omega)$ such that
\begin{gather} \label{Eq:Beta1}
\| \nabla \beta \|_{C^{k,\alpha}(K)} \leq C \varepsilon, \\
\label{Eq:Beta2}
\Delta \beta = \div (\tilde{u} - \nabla \theta) \text{ in } \Omega.
\end{gather}
To prove this we reason as in Step~3 with $\mbox{div}$ replacing $\mbox{curl}$. Namely, we extend $\div (\tilde{u} - \nabla \theta)_{|\overline{{\mathcal V}_{1}}}$ to some function $\rho \in C^{k,\alpha}_{c}(\R^{N})$ in such a way that
\begin{equation*}
\| \rho \|_{C^{k,\alpha}_{c}(\R^{N})} \leq C \| \div (\tilde{u} - \nabla \theta)_{|\overline{{\mathcal V}_{1}}} \|_{C^{k,\alpha}({\mathcal V}_{1})}.
\end{equation*}
Note that due to \eqref{H2} and since $u$ satisfies \eqref{Eq:StokesOmega} on ${\mathcal V}$, the right hand side is of order ${\mathcal O}(\varepsilon)$. Now there is a vector field $v_{\rho} \in C^{k+1,\alpha}(B(0,R))$ such that
\begin{equation*}
\div v_{\rho}= \rho \text{ in } B(0,R) \ \text{ and } \ \| v_{\rho} \|_{C^{k+1,\alpha}(B(0,R))} \leq C \| \rho \|_{C^{k,\alpha}_{c}(\R^{N})} \leq C \varepsilon.
\end{equation*}
Note that
\begin{equation} \label{Eq:Divvvrho}
\div (\tilde{u} - \nabla \theta - v_{\rho}) = 0 \text{ in } {\mathcal V}_{1}.
\end{equation}
We suppose here $N=3$. One can adapt the proof mutatis mutandis to $N=2$. 

Call $\Gamma_{1}, \dots, \Gamma_{K}$ the various connected components of $\partial {\mathcal V}_{1}$. Taking \eqref{Eq:Divvvrho} into account and due to the fact that each connected component of $\R^{N} \setminus {\mathcal V}_{1}$ meets at least a point $A_{i}$, we deduce that there exist $\lambda_{1}, \dots, \lambda_{n}$ so that
\begin{equation*}
\forall i=1, \dots , K, \ \  \int_{\Gamma_{i}} \left( \tilde{u} - \nabla \theta - v_{\rho} - \sum_{j=1}^{n} \lambda_{j} \frac{x-A_{j}}{|x-A_{j}|^{3}}\right) \cdot n =0.
\end{equation*}
We write
\begin{equation*}
h:= \sum_{j=1}^{n} \lambda_{j} \frac{x-A_{j}}{|x-A_{j}|^{3}}.
\end{equation*}
The functions defining $h$ are related to the {\it second} de Rham cohomology space of ${\mathcal V}_{1}$, denoted ${\mathcal H}^{2}({\mathcal V}_{1})$, which is a finite dimensional space and which can also be described as in \cite[Appendix I]{temam09}. The basic property of this space is: any divergence-free vector field can be written as a sum of a $\curl$ and an element of ${\mathcal H}^{2}({\mathcal V}_{1})$.
Now on ${\mathcal V}_{1}$, one has $\div(\tilde{u} -\nabla \theta-v_{\rho} -h)=0$ and 
\begin{equation*}
\int_{\Gamma} \left( \tilde{u} - \nabla \theta - v_{\rho} - h \right) \cdot n =0,
\end{equation*}
on any closed surface $\Gamma \subset {\mathcal V}_{1}$. It follows that that $\tilde{u} -\nabla \theta -v_{\rho}$ can be written in the form
\begin{equation} \label{Eq:VMVR}
\tilde{u} -\nabla \theta -v_{\rho} = \curl B + h \text{ in } {\mathcal V}_{1},
\end{equation}
for some $B \in C^{k+2,\alpha}({\mathcal V}_{1})$. In particular one has
\begin{equation} \label{Eq:EstB}
\| \tilde{u} -\nabla \theta -\curl B - h \|_{C^{k+1,\alpha}({\mathcal V}_{1})} = \| v_{\rho} \|_{C^{k+1,\alpha}(B(0,R))} \leq C \varepsilon.
\end{equation}
Now we extend $(\tilde{u} -\nabla \theta -\curl B - h)_{|{\mathcal V}_{1}}$ into some $\hat{U} \in C^{k+1,\alpha}(\Omega;\R^{3})$ and $B$ in some $\tilde{B} \in C^{k+2,\alpha}_{c}(B(0,R);\R^{3})$ in such a way that
\begin{equation} \label{Eq:EstHatw}
\| \hat{U} \|_{C^{k+1,\alpha}(\Omega)} \leq C \| \tilde{u} -\nabla \theta -\curl B - h \|_{C^{k+1,\alpha}({\mathcal V}_{1})}.
\end{equation}
We define
\begin{equation*}
\tilde{U}:=  \hat{U} + \curl \tilde{B} + h \text{ in } \Omega. 
\end{equation*}
We introduce $\tilde{q}$ as the solution in $\Omega$ of
\begin{equation*}
-\Delta \tilde{q} = \div \tilde{U} = \div \hat{U} \text{ in } \Omega, \ \ \tilde{q}=0 \text{ on } \partial \Omega,
 \end{equation*}
and $q$ as the solution in $\Omega$ of
\begin{equation*}
-\Delta q = \div (\tilde{u} - \nabla \theta) \text{ in } \Omega, \ \ \tilde{q}=0 \text{ on } \partial \Omega.
\end{equation*}
Note that thanks to \eqref{Eq:EstB} and \eqref{Eq:EstHatw}, we have
\begin{equation} \label{Eq:EstTildeq}
\| \tilde{q} \|_{C^{k+1,\alpha}(\Omega)} \leq C \varepsilon.
\end{equation}
Clearly, $\tilde{q}-q$ is harmonic on ${\mathcal V}_{1}$, so by Walsh's theorem one find some harmonic function $\psi$ in $\R^{N} \setminus \{ A_{1}, \dots, A_{n} \}$ such that
\begin{equation} \label{Eq:Estqqtilde}
\| \tilde{q}-q -\psi \|_{C^{k+1,\alpha}(K)} \leq \varepsilon.
\end{equation}
We let
\begin{equation*}
\beta := q-\psi.
\end{equation*}
One can easily check that \eqref{Eq:Beta2} holds, and \eqref{Eq:Beta1} is a consequence of \eqref{Eq:EstTildeq}, \eqref{Eq:Estqqtilde} and
\begin{equation*}
\| \beta \|_{C^{k+1,\alpha}(K)} \leq \| \tilde{q}-q -\psi \|_{C^{k+1,\alpha}(K)} + \| \tilde{q} \|_{C^{k+1,\alpha}(K)}.
\end{equation*}
\ \par
\noindent
{\bf Conclusion.} We define
\begin{equation*}
\overline{u} := \tilde{u} -\nabla \theta -\nabla \beta.
\end{equation*}
Due to \eqref{H2} and \eqref{Eq:Beta1} one clearly has
\begin{equation*}
\| u - \overline{u} \|_{C^{k,\alpha}(K)} \leq C \varepsilon.
\end{equation*}
Moreover
\begin{equation*}
\div \overline{u} =0 \text{ in } \Omega
\end{equation*}
follows from \eqref{Eq:Beta2}. Finally one has
\begin{equation*}
\Delta \overline{u} = - \nabla (\Delta(\theta + \beta)) \text{ in } \Omega,
\end{equation*}
thanks to \eqref{Eq:uepsilon1}. This concludes the proof of Theorem~\ref{Thm:RWforS}. \par
%
%

%
%
%
%
%
%
%
\section{An alternative approach} 
The above proof relies on the aforementioned Walsh theorem (Theorem~\ref{th:approxharm1}). The proof of Walsh's theorem is not really constructive; at least it seems difficult to imagine a numerical scheme which could be inherited from the proof. To take this into account, we propose another strategy for the subsequent application, based on a density argument which we hope can be used more easily for an approximation scheme. Of course, this density result can also be proved using Theorem~\ref{Thm:RWforS}. \par
We first introduce some definitions and notations. 
For $\Sigma$ a nonempty open part of the boundary of some regular open set $U$, we consider
%
$$H^{1/2}_m(\Sigma):=\{\phi \in H^{1/2}(\partial U;\R^{N}), \ \mbox{Supp\,} \phi \subset \Sigma, \ \int_\Sigma \phi \cdot n \, d\sigma=0 \}.$$
We also define
$$H^{-1/2}_m(\Sigma)=(H^{1/2}_m(\Sigma))'.$$
It is not difficult to see that for any $w \in H^{-1/2}_m(\Sigma)$, there is a unique $\tilde{w} \in H^{-1/2}(\Sigma;\R^{3})$ with $\langle \tilde{w} , n \rangle_{H^{-1/2} \times H^{1/2}} =0$ such that for any $u \in H^{1/2}_m(\Sigma)$, $\langle w, u \rangle = \langle \tilde{w} , u \rangle_{H^{-1/2} \times H^{1/2}}$. \par
%
%
%
%
%
The density result that we aim at proving is the following.
\begin{theorem} \label{th:dens}
Assume that $\gamma$ is a $C^\infty$ Jordan surface included in $\Omega$, then the set
\begin{equation*}
\big\{u_{|\gamma}, \ (u,p) \text{ is a solution to \eqref{eq:stokes} such that } u_{|\partial \Omega}\in H^{1/2}_m(\Sigma) \big\},
\end{equation*}
is dense in $H^{1/2}_m(\gamma)$.
\end{theorem}
\begin{proof}[Proof of Theorem~\ref{th:dens}]
We first assume that $\Sigma \cap \mbox{}^c\hbox{int}(\gamma)\neq \emptyset$.
Consider some $T\in H^{-1/2}(\gamma)$. To such $T$ we associate $T_{\gamma} \in {\mathcal D}'(\Omega)$ by
\begin{equation*}
\langle T_{\gamma}, \varphi \rangle_{{\mathcal D}'(\Omega) \times {\mathcal D}(\Omega)} := \langle T, \varphi_{|\gamma} \rangle_{H^{-1/2}(\gamma) \times H^{1/2}(\gamma)} .
\end{equation*}
It is not difficult to see that $T_{\gamma}$ is continuous on $H^{1}_{0}(\Omega)$, so that $T_{\gamma} \in H^{-1}(\Omega)$. 
Hence we can introduce $(v,q)\in H^1_0(\Omega)^N \times L^2(\Omega)$ the solution of 
\begin{subequations}\label{eq:adj}
\begin{align}
 -\Delta v + \nabla q & = T_{\gamma} \hbox{ in }\Omega, \label{eq:adj.1} \\
 \div v &= 0  \hbox{ in } \Omega, \\
 v & = 0 \hbox{ on }\partial \Omega, \\
 \int_\Omega q \, dx & = 0.
\end{align}
\end{subequations}
%
%
Consider now $u \in H^{1}(\Omega)$ a solution of \eqref{eq:stokes} with $u_{|\partial \Omega}\in H^{1/2}_m(\Sigma)$.
We perform the following computations:
%
%
\begin{eqnarray*}
T(u_{|\gamma}) &=& \int_\Omega (-\Delta v+\nabla q) \cdot u \, dx \\
&=& \int_\Omega -\Delta u \cdot v \, dx + \int_{\partial \Omega} \left(-\dfrac{\partial v}{\partial n}+q n\right) \cdot u \, d\sigma \\
&=& \int_{\partial \Omega} \left(-\dfrac{\partial v}{\partial n}+q n\right) \cdot u \, d\sigma
\end{eqnarray*}
We remark that the term $-\dfrac{\partial v}{\partial n}+q n$ in the right hand side of the last identity is defined in a classical way due to the local regularity of $(v,q)$ away from $\gamma$. \par
Now assume that there exists some $T\in H^{-1/2}(\gamma)$ such that 
$T(u_{|\gamma})=0$ for any $(u,p)$ solution of \eqref{eq:stokes} with $u_{|\partial \Omega}\in H^{1/2}_m(\Sigma)$. Then 
we can deduce that $-\dfrac{\partial v}{\partial n}+qn=0$ on $\Sigma$. 
Since moreover $v \in H^{1}_{0}(\Omega)$, by using the unique continuation property on the Stokes equation (see \cite{fabrelebeau} and \cite{bouleglogrand}) we deduce that $v=0$ and thus $q=0$ in $\Omega\setminus \hbox{int}(\gamma)$. \par
Since $v \in H^{1}(\Omega)$ and $q \in L^{2}(\Omega)$, we infer that $v_{|\hbox{int}(\gamma)}\in H_0^1(\hbox{int}(\gamma)\cap \Omega)$ and $q_{|\hbox{int}(\gamma)}\in L^2(\hbox{int}(\gamma)\cap \Omega)$. Since moreover one has
\begin{equation*}
-\Delta v+\nabla q=0 \ \text{ and } \ \div v =0 \  \text{ in } \mathcal{D}'(\hbox{int}(\gamma)\cap \Omega),
\end{equation*}
we can thus infer by uniqueness that $v=0$ and $q$ is constant in $\hbox{int}(\gamma)\cap \Omega$. Now if $q$ is a non null constant in $\hbox{int}(\gamma)\cap \Omega$ and while $q=0$ in $\Omega\setminus \hbox{int}(\gamma)$, this leads to the following form of $T$:
$$T(w)= c\int_\gamma w \cdot \nu \, d\sigma,$$
which means that $T=0$ as a member of $H^{-1/2}_{m}(\gamma)$. \par
To treat the case where $\Sigma \subset \hbox{int}(\gamma)$, we just have to exchange the role played by $\hbox{int}(\gamma)\cap \Omega$ and $\Omega\setminus \hbox{int}(\gamma)$.
This ends the proof of Theorem~\ref{th:dens}.
\end{proof}
%
%
%

%
%
%
%
%
%
\section{Proof of Theorem~\ref{th:2d3d}}
In this section, we establish Theorem~\ref{th:2d3d}. \par
We will use the following notation: for $\delta>0$ and $E$ a subset of $\BbR^p$, ${\mathcal V}_\delta(E)$ will denote a neighborhood of thickness $\delta$ of $E$:
\begin{equation*}
{\mathcal V}_{\delta}(E) = \{ x \in \Omega \ / \ d(x,E) < \delta \}.
\end{equation*}
In the sequel, when referring to smooth curves, surfaces or maps, we mean of class $C^\infty$. We will also use the equivalent real-analytic notions to which we refer by $C^\omega$. \par
\subsection{A model flow}
The first part of the proof consists in introducing a model flow, that is a solenoidal vector field that drives $\gamma_{0}$ to $\gamma_{1}$ exactly. But of course, it does not necessarily satisfy \eqref{eq:stokes2}. \par
For both cases $N=2,3$ the fact that $\gamma_0$ and $\gamma_1$ satisfy the property $\mathfrak{P}$ leads to the existence of a smooth isotopy preserving the volume given by the following theorem.
\begin{theorem}\label{th:exX}
Assume that $\gamma_0$ and $\gamma_1$ satisfy the property $\mathfrak{P}$, there exists $X\in C^\infty_0([0,1]\times \bar{\Omega}, \BbR^N)$ such that $\Div(X)=0$,
\begin{equation*}
\phi([0,1],0,\gamma_0)\subset \Omega \text{ and } \phi(1,0,\gamma_0)=\gamma_1.
\end{equation*}
\end{theorem}
In the general situation, this is a direct consequence of a result of Krygin \cite{krygin}.
An explicit proof of Theorem~\ref{th:exX} is given in \cite{GlHo08} when $N=2$. When $N=3$ an explicit but technical construction can also be made, relying on the construction of a ``pipe'' between $\gamma_0$ in $\gamma_1$. \par
%
%
%
%
%
%
%
%
%
\subsection{Extension of analytic solutions of the Stokes equation across the boundary}
As in \cite{GlHo08} and \cite{GlHo10}, an important case corresponds to the case when $\gamma_0$ and $X$ (such as given in Theorem~\ref{th:exX}) are analytic data. 
We show below that a solution of Stokes equation in an analytic domain with analytic data on the boundary can be extended across the boundary, which will allow to use Theorem~\ref{th:approxharm1}. 
\par
Consider $\gamma$ a $C^\omega$ Jordan surface in $\Omega$ and $X$ a $C^\omega$ divergence free vector field on $\Omega$. Depending on $\Sigma$ and $\gamma$, we introduce a new open set $\widetilde{\Omega}(\gamma)$ as follows:
\begin{itemize}
\item If $\gamma$ is contractible in $\Omega$, we set $\widetilde{\Omega}(\gamma):=\mathring{\hbox{int}(\gamma)}$.
\item If $\gamma$ is not contractible in $\Omega$ and if $\Sigma \subset \hbox{int}(\gamma)$, we set
$\widetilde{\Omega}(\gamma)=\Omega\setminus \hbox{int}(\gamma)$. 
\item If $\gamma$ is not contractible and if $\Sigma$ meets $\partial \Omega\setminus \hbox{int}(\gamma)$, we can always assume that $\Sigma \subset \partial \Omega\setminus \hbox{int}(\gamma)$, and we set $\widetilde{\Omega}(\gamma):= \Omega\cap \hbox{int}(\gamma)$.
\end{itemize}

We first recall the following result on the Stokes equation (see \cite{temam09} or \cite{boyerfabrie}).
\begin{proposition}
Let $\Omega$ a bounded regular domain in $\BbR^N$, $N=2$ or $3$. For any $u_0 \in H^{1/2}(\partial \Omega)^N$ such that
$$\int_{\partial \Omega}u_0 \cdot n \, d\sigma=0,$$
there exists a unique solution $(u,p)\in H^1(\Omega)^N \times L^2(\Omega)$ to \eqref{eq:stokes} with $u=u_{0}$ on $\partial \Omega$.
\end{proposition}
The following proposition proves our claim on the extension of analytic solutions across the boundary. 
\begin{proposition} \label{prop:ext}
Let $X\in C^\omega(\Omega)$ such that $\Div(X)=0$ in $\Omega$. 
Let $(u,p)$ be a solution of \eqref{eq:stokes.1}, \eqref{eq:stokes.2}, \eqref{eq:stokes.3} in $\widetilde{\Omega}(\gamma)$ such that 
$u=X$ on $\gamma$, and such that $u_{|\partial \widetilde{\Omega}(\gamma)\setminus\gamma}=0$ (this condition can be empty).
Then there exists a neighborhood $ \mathcal{V}_\delta(\gamma)$  such that $(u,p)$ can be extended to a solution of \eqref{eq:stokes.1}, \eqref{eq:stokes.2} in $\mathcal{V}_\delta(\gamma)\cup \widetilde{\Omega}(\gamma)$.
\end{proposition}
\begin{proof}[Proof of Proposition~\ref{prop:ext}]
This is a classical result. Indeed, for $N=2,3$, it is proven
in \cite{BOC-GUN95} (see also \cite{Costa-Dau-Nic} and \cite{GUO-SCH}) that the Stokes systems satisfies the
Agmon-Douglis-Niremberg ellipticity conditions (see \cite{AgDoNi1} and \cite{AgDoNi2}.)
Since $X$ is analytic and since $\gamma$ is analytic, the analyticity of $(u,p)$ up to $\partial \widetilde{\Omega}(\gamma)$ is then classical (see \cite[Section 6.6]{Mo08})).
This result also follows from the regularity results given in \cite{GUO-SCH}.
When $N=3$, one may use a similar argument as in \cite{Costa-Dau-Nic} as mentioned by \cite{Costa-priv}. \par
\end{proof}

We now follow the strategy given in \cite{GlHo08}.
\subsection{Proof of the main result when $X$ and $\gamma_0$ are analytic}
%
%
%
%
%
We first assume that $\gamma_0$ is real analytic and that $X\in C_0^1([0,1];C^\omega(\Omega))$. 
Due to the regularity of $X$, it follows that
\begin{equation*}
\gamma(t):= \phi^X(t,0,\gamma_0)
\end{equation*}
defines a continuous family of real analytic surfaces. As in Proposition~\ref{prop:ext}, we consider the solution  $(u(t),p(t))$ of
\begin{equation*}
\left\{ \begin{array}{l}
- \Delta u + \nabla p = 0 \text{ in } \widetilde{\Omega}(\gamma(t)), \\
\Div u =0 \text{ in } \widetilde{\Omega}(\gamma(t)), \\
u(t)=X(t) \text{ on } \gamma(t) \\
u(t)=0 \text{ on } \partial \widetilde{\Omega}(\gamma(t)) \setminus \gamma(t) .
\end{array} \right.
\end{equation*}
Then Proposition~\ref{prop:ext} which determines some neighborhood ${\mathcal V}_{\delta(t)}(\gamma(t))$ of $\gamma(t)$ such that $(u,p)$ extends to $V_{t} \cup \widetilde{\Omega}(\gamma(t))$. \par
Let us now point out that, due to the analyticity and compactness of $\gamma_0$, one can obtain a uniform size $\delta(t)$ of the neighborhood ${\mathcal V}_{\delta(t)}(\gamma(t))$ in $[0,T]$ and moreover get that $u \in C([0,1]; C^\infty(\widetilde{\Omega}(\gamma(t)) \cup {\mathcal V}_{\delta}(\gamma(t)))$ (with the obvious abuse of notations). 
First, since $\gamma$ and $X$ are continuous functions of time with values in real analytic functions in space and compose the data in Proposition~\ref{prop:ext}, we can deduce from (the proof of) Proposition~\ref{prop:ext} that for each $(t,x)$ in $[0,T] \times \gamma$, there is a neighborhood $U_{x}$ of $x$ and $(t-\eta,t+\eta)$ of $t$ such that $u$ can be extended in $U_{x}$ for all times in $(t-\eta,t+\eta)$. Using the compactness of $\gamma$ and the unique continuation principle, we can extend this to a whole neighborhood $\mathcal{V}_{\delta(t)}(\gamma(t))$ of $\gamma$ for times in $(t-\eta_{t},t+\eta_{t})$.
Now given $\eps>0$, using the compactness of $[0,T]$, we can obtain $0 \leq t_1 < \dots <t_n \leq 1$, $\eta_{1}, \dots, \eta_{n}>0$ and $\delta>0$ such that
\begin{gather}
\nonumber
[0,1] \subset \cup_{i=1}^n (t_i-\eta_{i}, t_i+ \eta_{i}), \\
\nonumber
\forall t \in [t_i-\eta_{i},t_i+\eta_{i}] \cap [0,1], \ \gamma(t) \subset {\mathcal V}_{\delta/2}(\gamma(t_i)), \\
\label{eq:continuiteunifX}
\forall s, t \in [t_i-\eta_i,t_i+\eta_i], \ \ \| u(s,\cdot) - u(t,\cdot) \|_{C^k(\overline{{\mathcal V}_{\delta/2}}(\gamma(t_i)))} \leq \eps,
\end{gather}
and
\begin{equation*}
\forall (t,s)\in (t_i-\eta_{i},t_i+\eta_{i}) \cap [0,1], \ \|X(t,\cdot)-X(s,\cdot)\|_{C^k({\mathcal V}_{\delta/2}(\gamma(t_i))} \leq \eps.
\end{equation*}
This includes in particular the uniformity that we claimed. More details in the harmonic case can be found in \cite[Proof of Lemma~3]{GlHo08}). \par
Now for each $i$, we determine Jordan surfaces $\widehat{\gamma}_i$ in $\Omega\setminus \widetilde{\Omega}(\gamma(t_i))\cap [{\mathcal V}_{3\delta/4}(\gamma(t_i)) \setminus \overline{\mathcal V}_{\delta/2}(\gamma(t_i))]$ such that $\gamma(t_i) \subset \widetilde{\Omega}(\widehat{\gamma}_i)$. The solution $(u(t_i),p(t_i))$ can hence be extended up to $\widehat{\gamma}_i$. \par
We now apply Theorem~\ref{Thm:RWforS} (one could use Theorem~\ref{th:dens} as well). Given $\eps>0$, we can find a solution $(U_i,P_i)$ of \eqref{eq:stokes} on $\Omega$ such that $U_{i|\partial \Omega} \in H^{1/2}_m(\Sigma)$ and such that
\begin{equation} \label{eq:est1}
\| U_i - u_i \|_{H^{1/2}_m(\widehat{\gamma}_i)} < \eps.
\end{equation}
By classical elliptic estimates, we infer from \eqref{eq:est1} that, for any $k$, there exists $C$ (not depending on $\eps$, but depending on $k$, $\widehat{\gamma}_i$ and $\delta$) such that
\begin{equation} \label{eq:est2}
\| U_i - u_i \|_{C^k(\overline{{\mathcal V}_{\delta/3}(\tilde{\Omega}(\gamma(t_i)))})} < C \eps.
\end{equation}
Take $(\kappa_i)_{i=1\dots n}$ a smooth partition of unity associated to the covering $(t_i-\eta_{i},t_i+\eta_{i})_{i=1\dots n}$ and consider 
\begin{equation*}
U(t,x) := \sum_{i=1}^n \kappa_i(t) U_i(x) \text{ and } P(t,x) := \sum_{i=1}^n \kappa_i(t) P_i(x). 	
\end{equation*}
Then for all $t\in [0,1]$ $(U(t,\cdot),P(t,\cdot))$ is a solution of \eqref{eq:stokes}, and we have
\begin{equation} \label{eq:est3}
\max_{t\in [0,1]} \| U - u \|_{C^k(\overline{{\mathcal V}_{\delta/3}(\tilde{\Omega}(\gamma(t)))})} < C \eps,
\end{equation}
for some $C$ independent of $\eps$. This implies that, provided $\eps$ is small enough,
\begin{equation} \label{eq:est3-bis}
\forall t \in [0,1], \ \|U\|_{C^k(\overline{\tilde{\Omega}(\gamma(t))})} < 1 + \|u\|_{C^k(\overline{\tilde{\Omega}(\gamma(t))})}.
\end{equation}
Note that $\phi^u(t,0,\gamma_0)=\phi^X(t,0,\gamma_0)$, and thus provided that $\eps>0$ is small enough we have, by Gronwall's lemma
\begin{equation*}
\|\phi^U(t,0,\gamma_0)-\phi^u(t,0,\gamma_0)\|_{\infty}\leq \| u - U \|_{C^0([0,1];C^0(\overline{V_{\eta/3}}(\gamma(t))))} e^{\|u\|_{L^1(0,1;W^{1,\infty}(V_{\eta/3}(\gamma(t))))}},
\end{equation*}
which proves in particular that the flow of $\gamma_0$ by $u$ remains in $\Omega$.
By classical regularity and using again Gronwall's lemma and \eqref{eq:est3}, we get
\begin{equation*}
\| \phi^U(t,0,\gamma_0) - \phi^u(t,0,\gamma_0) \|_{C^k(\mathbb{S}^{N-1})} \lesssim
\| u - U \|_{C^0([0,1];C^k(\overline{V_{\eta/3}}(\gamma(t))))}  e^{\|u\|_{L^1(0,1;W^{1+k,\infty}(V_{\eta/3}(\gamma(t))))}},
\end{equation*}
which ends the proof of Theorem~\ref{th:2d3d} in this situation. \par
\begin{remark}
If one wants to use Theorem~\ref{th:dens} rather than Theorem~\ref{Thm:RWforS} to get the ``extended approximation'' $(U_i,P_i)$, one chooses $K= \overline{\mathcal V}_{\delta/2}(\gamma(t_i)) \cup (\partial \Omega \setminus \Sigma)$ and one approximates the function given by $(u(t_{i}),p(t_{i}))$ on ${\mathcal V}_{\delta/2}(\tilde{\Omega}(\gamma(t_i)))$ and by $0$ on a neighborhood of $\partial \Omega \setminus \Sigma$. One does not get exactly $U_{i}=0$ on $\Sigma$, but one can remove a solution of \eqref{eq:stokes} given by Proposition~\ref{prop:ext} and associated with boundary conditions given by $U_{i}$ on $\partial \Omega \setminus \Sigma$ and to an extension of $U_{i}$ on $\Sigma$ of comparable size. This way we obtain a suitable approximation satisfying moreover the boundary conditions on $\partial \Omega \setminus \Sigma$.
\end{remark}
%
%
%
\subsection{Proof of the main result when $X$ is analytic and $\gamma_0$ is smooth.}
We will only consider the case of $N=3$, the case $N=2$ being similar.
In this part we assume that $X$ is a solenoidal vector field in $C([0,1]; C^\omega(\Omega))$, and that $\gamma_0$ is a smooth (but not necessarily analytic) embedded $2$-sphere in $\Omega$.
 \par
According to a result of Whitney (see \cite{Whi36}), $\gamma_0$ is embedded in a smooth family of smooth Jordan surfaces $\gamma_\nu,\ \nu \in (-\nu_0,\nu_0)$, such that for $\nu \neq 0$, $\gamma_\nu$ is analytic, and for $\nu\neq \nu'$ we have $\gamma_\nu \cap \gamma_{\nu'}=\emptyset$. We can assume that $\gamma_0\subset \hbox{int}(\gamma_\nu)$ for $\nu>0$. \par
Now for each $\nu>0$ and any $t\in [0,1]$, we consider $(u_\nu(t),p_\nu(t))$ given by Proposition~\ref{prop:ext} with $\phi^{X}(t,0,\gamma_\nu)$ instead of $\gamma$. Then given $\eps>0$, we determine $(U_{\eps,\nu}(t),P_{\eps,\nu}(t))$ such that \eqref{eq:approxlagrcontr2d} holds with $\gamma_\nu$ instead of $\gamma_0$.
 
Let us point out that, using elliptic regularity (noting that the constant in the elliptic regularity estimate is uniform in $\nu$ due to $\gamma_{\nu} \rightarrow \gamma_{0}$ in $C^{\infty}$ as $\nu \rightarrow 0$), \eqref{eq:continuiteunifX} and \eqref{eq:est3-bis}, we have uniform bounds with respect to $\nu$ on $\| U_{\eps,\nu} \|_{C([0,1];C^k(\phi^{U_{\eps,\nu}}(t,0,\gamma_\nu))}$. \par
Therefore we get by Gronwall's lemma that
\begin{multline*}
\| \phi^{U_{\eps,\nu}}(t,0,\gamma_\nu) - \phi^{U_{\eps,\nu}}(t,0,\gamma_0) \|_{C^k(\mathbb{S}^2)} \leq \\
C_{k} \| \gamma_0 - \gamma_{\nu}\|_{C^k(\mathbb{S}^2)}
\exp \left( \int_0^1 \| U_{\eps,\nu} \|_{C^{k+1}(\hbox{int} (\phi^{U_\eps}(t,0,\gamma_{\nu}))} \, dt \right),
\end{multline*}
which proves the result in this situation.
\subsection{Proof of the main result in the general case}
Again we will consider the case $N=3$. We assume now that we are in the general case, that is, $X$ is merely $C^{\infty}(\overline{\Omega};\BbR^{3})$. Let $\lambda>0$ such that
\begin{equation*}
\max_{t \in [0,1]} \mbox{dist}(\phi^{X}(t,0,\gamma_{0}) , \partial \Omega) >2 \lambda.
\end{equation*}
We define ${\mathcal U}_{t}:= \overline{\mathcal V}_{\lambda}(\hbox{int}(\gamma(t)))$. Reducing $\lambda$ if necessary, we can obtain that for all $t$, ${\mathcal U}_{t}$ is diffeomorphic to a ball. \par
Now we can use the Whitney's approximation theorem (see e.g. \cite[Proposition 3.3.9]{92KRAPRA}),
for any $\mu>0$ and any $k \in \BbN$ there exists $X_{\mu} \in C([0,1];C^\omega(\BbR^3))$ such that
\begin{equation*}
\| X_{\mu} - X \|_{C([0,1];C^{k+1}({\mathcal U}_{t}))} \leq \mu.
\end{equation*}
Moreover, we can ask that
\begin{equation*}
\div X_{\mu} =0 \ \text{ in } \ [0,1] \times \BbR^{3}.
\end{equation*}
This is just a matter of writing $X$ in the form $X=\curl A$ (using the fact that ${\mathcal U}_{t}$ is a topological ball) and approximating $A$ at order $k+2$. \par
%
%
Using as before the compactness of the time interval $[0,1]$ and a partition of unity (as for the proof when $X$ and $\gamma_0$ are analytic), we can obtain a smooth approximation uniformly in time. \par 
We then apply Gronwall's lemma to infer that
\begin{equation*}
\| \Phi^{X}(t,0,\gamma_{0}) - \Phi^{X_{\mu}}(t,0,\gamma_{0}) \|_{k} \leq
\| X - X_{\mu} \|_{C^{0}([0,1]; C^{k}(\overline{\Omega})} \exp(\| X \|_{L^{1}(0,1;C^{k+1}(\overline{\Omega}))}),
\end{equation*}
and apply the preceding procedure when $\gamma_0$ is smooth and $X$ is analytic. \par
This concludes the proof of Theorem~\ref{th:2d3d}.
%
%
%
%
%
%
\section{A final remark}
One may wonder if it is possible to take into account a non-trivial initial condition $(u,p)_{|t=0}=(u_0,p_0)$ at $t=0$, where $(u_0,p_0)$ is a solution of \eqref{eq:stokes}. The question is whether it is still possible to solve the problem of approximate Lagrangian controllability, that is, to find $(u,p)$ a time dependent solution of \eqref{eq:stokes} such that \eqref{eq:approxlagrcontr2d} holds and moreover satisfying $u_{|t=0} = u_0$ and $p_{|t=0} = p_0$. \par
This is actually an easy consequence of Theorem~\ref{th:2d3d}. In order to do so, we assume that $(u_0,p_0)$ is a solution of \eqref{eq:stokes} such that $u_0\in \mathcal{L}ip(\Omega)$. Consider $\tau>0$ such that for $t\in [0,\tau]$, $\phi^{u_0}(t,0,\tau)\subset \Omega$. Note that this always possible due to the regularity of $u_0$.
Define $(\widetilde{u}_0(t),\widetilde{p}_0(t))=\dfrac{(\tau-t)^2}{\tau^2}(u_0,p_0)$ then $\forall t\in [0,\tau]$, 
$\phi^{\widetilde{u}_0}(t,0,\gamma_0) \subset \Omega$ and we have $\widetilde{u}_0(\tau)=0$. We can then apply the same procedure as before and obtain
\begin{theorem}
The results of Theorem~\ref{th:2d3d} remain true when we impose $u_{|t=0} = u_{0} \in \mathcal{L}ip(\Omega)$, with the constraint that $\div u_{0} =0$ and $u_{0|\partial \Omega \setminus \Sigma}=0$.
\end{theorem}
\bibliographystyle{plain}
\bibliography{Biblio-horsin}

 \end{document}